\documentclass[review]{elsarticle}
\usepackage{amsmath}
\usepackage{graphicx,setspace}
\usepackage{amsfonts,amsmath, amssymb, amsthm}
\usepackage{mathrsfs}
\usepackage{epstopdf}
\usepackage{float}
\usepackage{lineno,hyperref}
\usepackage{color}
\usepackage{subfigure}
\usepackage{bm}
\usepackage{graphicx}
\usepackage{amssymb, amsthm}
\usepackage{mathrsfs}
\usepackage{epstopdf}
\usepackage{float}
\usepackage{lineno,hyperref}
\usepackage{color}
\usepackage{subfigure}
\usepackage{cases}

\modulolinenumbers[5]
\numberwithin{equation}{section}
\begin{document}

\newtheorem{definition}{Definition}
\newtheorem{lemma}{Lemma}
\newtheorem{remark}{Remark}
\newtheorem{theorem}{Theorem}
\newtheorem{proposition}{Proposition}
\newtheorem{assumption}{Assumption}
\newtheorem{example}{Example}
\newtheorem{corollary}{Corollary}
\def\e{\varepsilon}
\def\Rn{\mathbb{R}^{n}}
\def\Rm{\mathbb{R}^{m}}
\def\E{\mathbb{E}}
\def\hte{\bar\theta}
\def\cC{{\mathcal C}}
\numberwithin{equation}{section}

\begin{frontmatter}

\title{Kantorovich-Rubinstein Distance and Approximation for   Non-local Fokker-Planck Equations}
\author{\bf\normalsize{
Ao Zhang\footnote{School of Mathematics and Statistics, \& Center for Mathematical Sciences,  Huazhong University of Sciences and Technology, Wuhan 430074,  China. Email: \texttt{zhangao1993@hust.edu.cn}} and Jinqiao Duan\footnote{Departments of Applied Mathematics \& Physics, Illinois Institute of Technology, Chicago, IL 60616, USA. Email: \texttt{duan@iit.edu}}
}}

\begin{abstract}
This work is devoted to studying complex dynamical systems under non-Gaussian fluctuations.
We first estimate the Kantorovich-Rubinstein distance for solutions of non-local Fokker-Planck equations  associated with stochastic differential equations with non-Gaussian L\'evy noise. This is then  applied to establish  weak convergence  of the corresponding probability distributions. Furthermore, this leads to smooth approximation for  non-local Fokker-Planck equations, as illustrated in an example.
\end{abstract}

\journal{Chaos}


\end{frontmatter}

\textbf{Nonlinear systems are often under non-Gaussian random fluctuations. The probability distribution evolution is for such a stochastic dynamical system is described by a non-local Fokker-Planck equation. This work is devoted to studying complex dynamical systems under non-Gaussian fluctuations. In this present paper, we   estimate the Kantorovich-Rubinstein distance  for solutions of non-local Fokker-Planck equations  associated with stochastic differential equations with non-Gaussian L\'evy noise. As a consequence, this provides a smooth approximation for  non-local Fokker-Planck equations.
}

\section{Motivation}

Complex systems are often under the influence of randomness \cite{Moss1, Horst, Gar, VanKampen3}. Uncertainties may also be caused by our lack of knowledge of some physical processes that are not well represented in the mathematical models \cite{Palmer1, ChenDuan, Kantz, Wilks, Williams}. Although these random mechanisms appear to be very small or very fast, their long time impact on the system evolution may be delicate or even profound \cite{Arnold, DuanBook2015}. These delicate impacts on the overall evolution of dynamical systems has been observed in, for example, stochastic bifurcation \cite{Crauel, CarLanRob01, Horst}, stochastic resonance \cite{imkeller2002model}, and noise-induced pattern formation \cite{Gar, blomker2003pattern}. Hence taking stochastic effects into account is of central importance for mathematical modeling of complex systems under uncertainty, and this leads to stochastic ordinary differential equations (SDEs) \cite{Arnold, Ikeda, Okse2003, WaymireDuan}. It is therefore crucial to investigate dynamics under uncertainty, in the context of models arising from applications in, for example, geophysical-climate systems.

Fluctuations in complex systems are often non-Gaussian \cite{Woy, Dit, Swinney, Shlesinger, taqqu, dybiec2009levy} rather than Gaussian. For instance, it has been argued that diffusion by geophysical turbulence \cite{Shlesinger} corresponds to a series of ``pauses", when the particle is trapped by a coherent structure, and ``fights" or ``jumps" or other extreme events, when the particle moves in the jet flow. Paleoclimatic data \cite{Dit} also indicate such irregular processes. There are also experimental demonstrations of L\'evy fights in optimal foraging theory and rapid geographical spread of emergent infectious disease.   This motivates our investigation of dynamical systems driven by non-Gaussian fluctuations.

 We consider the following stochastic differential equation (SDE) in $\mathbb{R}^{d}$ with a (non-Gaussian) L\'evy process:
\begin{equation}\label{SDE}
\mathrm{d} X_{t}=b_{t}\left(X_{t}\right) \mathrm{d} t+\int_{|z|<R} g_{t}\left(X_{t-}, z\right) \tilde{N}(\mathrm{d}t, \mathrm{d}z),
\end{equation}
where the drieft (i.e., vector field) $b: [0, T] \times \mathbb{R}^{d} \rightarrow \mathbb{R}^{d}$ and the `anomalous diffusion'  (i.e., jump coefficient)  $g: [0, T] \times \mathbb{R}^{d} \times \mathbb{R}^{d} \rightarrow \mathbb{R}^{d}$ are Borel measurable functions.  Here   $N$ is a Poisson random measure with intensity measure $\mathrm{d} t \nu(\mathrm{d}z)$, and $\tilde{N}(\mathrm{d}t, \mathrm{d} z):=N(\mathrm{d}t, \mathrm{d}z)-\mathrm{d}t \nu(\mathrm{d}z)$ is the compensated Poisson random measure \cite{DuanBook2015, Applebaum}. Moreover, $R > 0$ is a fixed constant, and $T$ is the duration for the system evolution. 
There are recent works on SDEs with jumps, as in, for example,  \cite{DuanBook2015, QZ08, XZ20, XZ13, K96}.

The non-local Fokker-Planck equation for the probability distribution (i.e., probability measure)  $\mu_{t}$   associated to SDE \eqref{SDE} is
\begin{equation}\label{FPE}
\partial_{t} \mu_{t}=\left(\mathscr{L}^{b}+\mathscr{L}_{\nu}^{g}\right)^{*} \mu_{t},  \quad \mu|_{t=0}=\mu_0,
\end{equation}
where
\begin{equation*}
\mathscr{L}^{b} u(x):=\left(b_{t}^{i} \partial_{i} u\right)(x), \quad
\mathscr{L}_{\nu}^{g} u(x):=\int_{|z|<R}\left[u(x+g_{t}(x, z))-u(x)-g_{t}(x, z) \cdot \nabla u(x)\right] \nu(\mathrm{d}z).
\end{equation*}
Here and below, we use Einstein's convention that the repeated indices in a product will be summed automatically.

Recall that the usual (or local) Fokker-Planck equation \cite{BKR15, BRS15, GW12} for an SDE with (Gaussian) Brownian motion $ dX_t=b(X_t) dt + \sigma(X_t) dB_t $  has been actively studied from the functional-analytical, variational and as well from the probabilistic point of view.
 
 In this present paper, we will estimate the distance between solutions (as probability distributions) of the non-local Fokker-Planck equation (\ref{FPE}), as  this distance plays an important role not only for the study of such qualitative properties of solutions as uniqueness and stability, but also for numerical approximations of probability distributions. The estimates for distances between solutions to non-local Fokker-Planck equations with different drift or noise terms are particularly interesting.

If the jump coefficient $g$ is identically zero, then non-local Fokker-Planck equation \eqref{FPE} reduces to the continuity equation (or Liouville equation)
\begin{equation}\label{CE}
\partial_{t} \mu_{t}+\operatorname{div}\left(\mu_{t} b_{t}\right)=0, \quad \mu|_{t=0}=\mu_0.
\end{equation}
Seis \cite{CS17} recently established quantitative stability estimates for solutions of \eqref{CE} in terms of the Kantorovich-Rubinstein distance. More precisely, for each $i=1,2$, if $b^{i} \in L^{1}\left(0, T ; W^{1, p}\left(\mathbb{R}^{d}, \mathbb{R}^{d}\right)\right)(p>1)$ and $\mu_{t}^{i}=\rho_{t}^{i} \mathrm{d} x$ is a solution to the continuity equation \eqref{CE} such that density $\rho^{i} \in L^{\infty}\left(0, T ; L^{1} \cap L^{p /(p-1)}\left(\mathbb{R}^{d}\right)\right)$, then
\begin{equation*}
\mathcal{D}_{\delta}\left(\mu_{t}^{1}, \mu_{t}^{2}\right) \leq C_{1}+\frac{C_{2}}{\delta}\left\|b^{1}-b^{2}\right\|_{L^{1}\left(0, T ; L^{p}\left(\mathbb{R}^{d}\right)\right)},
\end{equation*}
where $C_{1}$ and $C_{2}$ are constants depending on norms of vector fields $b^{1}, b^{2}$ and solutions $\rho^{1}, \rho^{2}$ (see Remark \ref{rem} below for the definition of $\mathcal{D}_{\delta}$ ). There are also results on the qualitative stability of \eqref{CE}; roughly speaking, if the sequence of vector fields $b^{n}$ converge to some $b$ in a certain sense, then the corresponding solutions $\mu^{n}$ tend to $\mu$ too.

Recently, Li and Luo \cite{LL19} considered a local Fokker-Planck equation
\begin{equation}\label{FPE1}
\partial_{t} \mu_{t}-\frac{1}{2} \sum_{k, l=1}^{d} \partial_{k l}\left(\mu_{t} a_{k l}\right)+\operatorname{div}\left(\mu_{t} b_t\right)=0, \quad \mu|_{t=0}=\mu_0,
\end{equation}
and establish quantitative stability estimates for the FPE \eqref{FPE1}: for each $i=1,2$, if $b^{i} \in L^{1}\left(0, T; W^{1, p}\left(\mathbb{R}^{d}, \mathbb{R}^{d}\right)\right)$, $\sigma^{i} \in L^{2}\left(0, T; W^{1,2 p}\left(\mathbb{R}^{d}, \mathcal{M}_{d, m}\right)\right)$ ($p>1$), where $\mathcal{M}_{n, m}$ is the space of $n \times m$ matrices, $a=\sigma \sigma^{*}$, and $\mu_{t}^{i}=\rho_{t}^{i} \mathrm{d}x$ is a solution to the corresponding Fokker-Planck equation \eqref{FPE1} such that $\rho^{i} \in L^{\infty}\left(0, T ; L^{1} \cap L^{p /(p-1)}\left(\mathbb{R}^{d}\right)\right)$, then
\begin{equation*}
\begin{aligned}
\tilde{\mathcal{D}}_{\delta}\left(\mu_{t}^{1}, \mu_{t}^{2}\right) \leq & \tilde{\mathcal{D}}_{\delta}\left(\mu_{0}^{1}, \mu_{0}^{2}\right)+C_1\left(\frac{1}{\delta}\left\|b^{1}-b^{2}\right\|_{L^{1}\left(L^{p}\right)}+\frac{1}{\delta^{2}}\left\|\sigma^{1}-\sigma^{2}\right\|_{L^{2}\left(L^{2 p}\right)}^{2}\right) \\
&+C_2\left(\left\|\nabla b^{1}\right\|_{L^{1}\left(L^{p}\right)}+\left\|\nabla \sigma^{1}\right\|_{L^{2}\left(L^{2 p}\right)}^{2}\right),
\end{aligned}
\end{equation*}
where $C_{1}$ and $C_{2}$ are constants depending on $d, p$ and norms of solutions $\rho^{1}, \rho^{2}$ (see \eqref{KR} below for the definition of $\tilde{\mathcal{D}}_{\delta}$ ).

Manita \cite{OM} drived estimates for the Kantorovich transportation cost functional along solutions to different Fokker-Planck equations \eqref{FPE1} for measures with the same uniformly elliptic and bounded diffusion coefficient but different dissipative drifts and different initial conditions. Bogachev, R\"ockner and Shaposhnikov \cite{BRS16} obtained upper bounds for the total variation, entropy and Kantorovich distances between two solutions to local Fokker-Planck equations \eqref{FPE1}, under the main assumptions that the diffusion coefficients are locally uniformly elliptic and locally Lipschitz continuous in spatial variables, and the drift coefficients are locally bounded.

In this present paper,  we will hence investigate the Kantorovich-Rubinstein distance for solutions of non-local Fokker-Planck equations \eqref{FPE}.  Using the superposition principle recently proved in \cite{FX19, XZ13}, we will obtain the Kantorovich-Rubinstein distance for non-local Fokker-Planck equations \eqref{FPE} with weakly differentiable coefficients and examine the consequences (stability and smooth approximation) of this distance estimation.

This paper is organized as follows. In Section \ref{sec2} we introduce some spaces and notations, define the Kantorovich-Rubinstein distance, and state our main result. In Section \ref{sec3}, we make the necessary preparations for proving the main results. The proof of the main result Theorem \ref{theo} is in Section \ref{sec4} and its consequence in Remark \ref{rem1}. This paper ends with illustrative examples for smooth approximation for non-local Fokker-Planck equation and a brief discussion in Section \ref{sec5}.

\section{Statement of Main Result}\label{sec2}
Before state our main results, we introduce some spaces and notations.

Let $(\Omega, \mathscr{F},\left(\mathscr{F}_{t}\right)_{t \geq 0}, \mathbb{P})$ be a complete filtered probability space.

For $p, q \in[1, \infty]$ and $0 <T<\infty$, let $L^{q}\left(0, T ; L^{p}\left(\mathbb{R}^{d}\right)\right)$ be the space of all Borel functions on $[0, T] \times \mathbb{R}^{d}$ with norm
\begin{equation*}
\|f\|_{L^q(L^p)}:=\left(\int_{0}^{T}\left(\int_{\mathbb{R}^{d}}|f(t, x)|^{p} \mathrm{d} x\right)^{q / p} \mathrm{d} t\right)^{1 / q}<\infty.
\end{equation*}
For $p=\infty$ or $q=\infty,$ the above norm is understood as the usual $L^{\infty}$-norm. For a measurable function $g_{t}(x, z): \mathbb{R}_{+} \times \mathbb{R}^{d} \times \mathbb{R}^{d} \rightarrow \mathbb{R}^{d}$, we define the following quantity which will be used below: for $j=0,1$
\begin{equation*}
{\Gamma_{j, R}^{q, p}}(g):=\int_{|z|<R}\left\|\nabla_{x}^{j} g(z)\right\|^{q}_{L^q(L^p)} \nu(\mathrm{d}z).
\end{equation*}

Let $\mathcal{P}\left(\mathbb{R}^{d}\right)$ be the space of all probability measures on $\mathbb{R}^{d}$ endowed with the weak convergence topology. For every $\mu, \nu \in \mathcal{P}\left(\mathbb{R}^{d}\right),$ let $\mathcal{C}(\mu, \nu)$ be the collection of probability measures on $\mathbb{R}^{d} \times \mathbb{R}^{d}$ having $\mu$ and $\nu$ as the first and the second marginal distribution, respectively. Fix some $\delta>0$, define the Kantorovich-Rubinstein distance with logarithmic cost function taken from the theory of optimal transportation and given by
\begin{equation}\label{KR}
\tilde{\mathcal{D}}_{\delta}(\mu, \nu)=\inf _{\pi \in \mathcal{C}(\mu, \nu)} \int_{\mathbb{R}^{d} \times \mathbb{R}^{d}} \log \left(\frac{|x-y|^{2}}{\delta^{2}}+1\right) \mathrm{d} \pi(x, y).
\end{equation}

According to the elementary inequality
\begin{equation*}
\log \left(1+s^{2}\right) \leq \log \left(1+2 s+s^{2}\right)=2 \log (1+s), \quad s \geq 0,
\end{equation*}
the quantity $\tilde{\mathcal{D}}_{\delta}(\mu, \nu)$ is finite if $\mu, \nu \in \mathcal{P}_{\log }\left(\mathbb{R}^{d}\right),$ where
\begin{equation*}
\mathcal{P}_{\log }\left(\mathbb{R}^{d}\right)=\left\{\nu \in \mathcal{P}\left(\mathbb{R}^{d}\right): \int_{\mathbb{R}^{d}} \log (1+|x|) \mathrm{d} \nu(x)<\infty\right\}.
\end{equation*}
Note that $\log \left(1+|x-y|^{2} / \delta^{2}\right)$ does not define a metric on $\mathbb{R}^{d}$.

\begin{remark}\label{rem}
(1) The Kantorovich-Rubinstein distance used in \cite{CS17} is defined as
\begin{equation*}
\mathcal{D}_{\delta}(\mu, \nu)=\inf _{\pi \in \mathcal{C}(\mu, \nu)} \int_{\mathbb{R}^{d} \times \mathbb{R}^{d}} \log \left(\frac{|x-y|}{\delta}+1\right) \mathrm{d} \pi(x, y), \quad \mu, \nu \in \mathcal{P}_{\log }\left(\mathbb{R}^{d}\right).
\end{equation*}
The two quantities $\mathcal{D}_{\delta}(\mu, \nu)$ and $\tilde{\mathcal{D}}_{\delta}(\mu, \nu)$ have the following relations:
\begin{equation*}
\tilde{\mathcal{D}}_{\delta}(\mu, \nu) \leq 2 \mathcal{D}_{\delta}(\mu, \nu) \quad \text { and } \quad \mathcal{D}_{\delta}(\mu, \nu) \leq\left(\frac{\tilde{\mathcal{D}}_{\delta}(\mu, \nu)}{\log 2}\right)^{1 / 2}+\tilde{\mathcal{D}}_{\delta}(\mu, \nu).
\end{equation*}
(2) For every $\mu, \nu \in \mathcal{P}_{\log }\left(\mathbb{R}^{d}\right),$ if
\begin{equation*}
\lim _{\delta \rightarrow 0} \frac{\tilde{\mathcal{D}}_{\delta}(\mu, \nu)}{|\log \delta|}=0,
\end{equation*}
then $\mu=\nu$. This fact is useful in many applications, for instance, the proof of uniqueness of solutions to the non-local Fokker-Planck equation \eqref{FPE}.\\
(3) The reason for us to choose $\tilde{\mathcal{D}}_{\delta}$ instead of $\mathcal{D}_{\delta}$ is that the square function is convenient for applying the It\^o formula.
\end{remark}

The following convention will be used throughout the paper: $C$ with or without indices will denote different positive constants (depending on the indices) whose values may change from one instance to another.

We consider the Kantorovich-Rubinstein distance for solutions of non-local Fokker-Planck equations with weakly differentiable coefficients. The main theorem is as follows.

\begin{theorem}\label{theo}
 (Upper bound for the  Kantorovich-Rubinstein distance of probability distributions)\\
Let $p>1$ and $q$ be such as $1/p + 1/q=1 $. For each $i \in\{1,2\}$, assume that $g^{i}$ satisfies ${\Gamma_{j, R}^{2, 2p}}(g^i)+{\Gamma_{j, R}^{4, 4p}}(g^i)<\infty$ for all $j=0, 1$, and $b^{i} \in L^{1}\left(0, T ; W^{1, p}\left(\mathbb{R}^{d}, \mathbb{R}^{d}\right)\right),$ and the probability density $\rho^{i} \in L^{\infty}\left(0, T ; L^{1} \cap L^{q}\left(\mathbb{R}^{d}\right)\right)$ with probability distribution $\mu_{t}^{i}=\rho_{t}^{i} \mathrm{d} x$ be the solution to the corresponding non-local Fokker-Planck equation \eqref{FPE} with $b=b^{i}$ and $g=g^i$.   Assume that the initial probability distribution $\mu_{0}^{i} \in \mathcal{P}_{\log }\left(\mathbb{R}^{d}\right), i=1, 2.$    \\ Then, for all $t \in[0, T]$ we have the following Kantorovich-Rubinstein distance estimate 
\begin{equation}\label{theo1}
\begin{aligned}
\tilde{\mathcal{D}}_{\delta}\left(\mu_{t}^{1}, \mu_{t}^{2}\right) & \leq \tilde{\mathcal{D}}_{\delta}\left(\mu_{0}^{1}, \mu_{0}^{2}\right)+C\|\rho^{2}\|_{L^{\infty}(L^{q})}\left[ \frac{1}{\delta}\|b^{1}-b^{2}\|_{L^{1}(L^{p})}+\frac{1}{\delta^{2}} \int_{|z|<R} \|g^{1}-g^{2}\|_{L^{2}(L^{2 p})}^{2}\nu(\mathrm{d}z) \right.\\
&\left. +\frac{1}{\delta^{4}} \int_{|z|<R} \|g^{1}-g^{2}\|_{L^{4}(L^{4 p})}^{4}\nu(\mathrm{d}z) \right] +C^{\prime}\left(\sum_{i=1}^{2}\|\rho^{i}\|_{L^{\infty}(L^{q})}\right)\times \\
& \left[\|\nabla b^{1}\|_{L^{1}(L^{p})} + \int_{|z|<R} \|\nabla g^{1}\|_{L^{2}(L^{2 p})}^{2}\nu(\mathrm{d}z)+ \int_{|z|<R} \|\nabla g^{1}\|_{L^{4}(L^{4 p})}^{4}\nu(\mathrm{d}z) \right].
\end{aligned}
\end{equation}
Here $C$ and $C^{\prime}$ are positive constants depending only on $d$ and $p$.
\end{theorem}

\begin{remark}\label{rem1}
The quantity $\tilde{\mathcal{D}}_{\delta}(\mu, \nu)$ is a   distance (or metric) in the space of configurations of same total mass \cite[Theorem 7.3]{CV03}, and it metrizes weak convergence \cite[Theorem 7.12]{CV03}: If there exists a sequence of $\delta_{k}$'s decaying to zero as $k \rightarrow \infty$ and such that $\tilde{\mathcal{D}}_{\delta}(\mu_k, \mu)$ is uniformly bounded, then the probability distributions $\mu_{k}$ converges weakly to $\mu$,  with a rate not larger than $\delta_{k}$. In the illustrative examples in the final section, we will see this leads to a smooth approximation for a non-local Fokker-Planck equation.\\
\end{remark}

\section{Preparations}\label{sec3}

We present some preparations   for the proofs of   Theorem 1.
\begin{definition}
A family of probability measures $\left(\mu_{t}\right)_{t \in[0, T]} \subset \mathcal{P}\left(\mathbb{R}^{d}\right)$ is called a weak solution of non-local Fokker-Planck equation \eqref{FPE} if\\
(i) for a.e. $t \in(0, T)$, $\mu_t$ is absolutely continuous with respect to the Lebesgue measure $dx$;\\
(ii) it holds
\begin{equation}\label{def1.1}
\int_{0}^{T} \int_{\mathbb{R}^{d}}\left(\left|b_{t}\right|+\int_{|z|<R}|g_t(x, z)|^2\nu(\mathrm{d} z)\right) \mathrm{d} \mu_{t} \mathrm{d} t<\infty;
\end{equation}
(iii) for any $f \in C_{c}^{1,2}\left([0, T) \times \mathbb{R}^{d}\right)$, one has
\begin{equation}\label{def1.2}
\int_{\mathbb{R}^{d}} f(0, x) \mathrm{d} \mu_0(x)+\int_{0}^{T} \int_{\mathbb{R}^{d}}\left(\partial_{t} f(t, x)+(\mathscr{L}^{b}+\mathscr{L}_{\nu}^{g}) f(t, x)\right) \mathrm{d} \mu_{t}(x) \mathrm{d} t=0,
\end{equation}
where $\mu_0 \in \mathcal{P}\left(\mathbb{R}^{d}\right)$.
\end{definition}

For $r>0$, let $|B(r)|$ be its Lebesgue measure, for the ball $B(r):=\left\{x \in \mathbb{R}^{d}:|x|<r\right\}$. For a locally integrable function $\varphi$ on $\mathbb{R}^{d}$, the Hardy-Littlewood maximal function is defined by
\begin{equation}\nonumber
M \varphi(x):=\sup_{r>0} \frac{1}{|B(r)|} \int_{B(r)}|\varphi(x+y)| \mathrm{d}y
\end{equation}
Then, the following lemma is well known in harmonic analysis and can be found in \cite[Lemma 3.1]{LL19} and \cite[Lemma 3.5, 3.7]{Z13}.

\begin{lemma}\label{lem}
(i) Let $\varphi \in W_{loc}^{1,1}(\mathbb{R}^{d})$. Then there exist a constant $C_{d}>0$ and a negligible set A such that for all $x, y \in A^c$,
\begin{equation}\label{lem1}
|\varphi(x)-\varphi(y)| \leq C_{d}|x-y|(M|\nabla \varphi|(x)+M|\nabla \varphi|(y)).
\end{equation}
(ii) Let $\varphi \in L^{p}(\mathbb{R}^{d})$ for every $p>1$. Then there exists a constant $C_{d, p}>0$ such that
\begin{equation}
\int_{\mathbb{R}^{d}}(M \varphi(x))^{p} \mathrm{d} x \leq C_{d, p} \int_{\mathbb{R}^{d}}|\varphi(x)|^{p} \mathrm{d}x.
\end{equation}
\end{lemma}
\begin{remark}
If $\varphi$ is continuous, then \eqref{lem1} holds for all $x, y \in \mathbb{R}^{d}$ (cf. \cite[Appendix]{FL10}).
\end{remark}

Let $\mathbb{D}_T=\mathbb{D}\left([0, T], \mathbb{R}^{d}\right)$ be the space of all $\mathbb{R}^{d}$-valued c\`adl\`ag functions on $[0, T]$, endowed with the Skorokhod topology so that $\mathbb{D}_T$ becomes a Polish space. Let $X_{t}(\omega)=\omega_{t}$ be the canonical process. For $t \in [0, T]$, let $\mathcal{B}_{t}^{0}(\mathbb{D}_T)$ denote the natural filtration generated by $\left(X_{s}\right)_{s \in[0, t]}$, and let
\begin{equation*}
\mathcal{B}_{t}:=\mathcal{B}_{t}(\mathbb{D}_T):=\cap_{s \geq t} \mathcal{B}_{t}^{0}(\mathbb{D}_T).
\end{equation*}
Now we recall the notion of martingale solutions associated with $\mathscr{L}_{t}$ in the sense of Stroock-Varadhan \cite{SV06}.
\begin{definition}
{(Martingale solution)}
Let $\mu_{0} \in \mathcal{P}(\mathbb{R}^{d})$. We call a probability measure $\mathbb{P}_{\mu_0} \in \mathcal{P}(\mathbb{D}_T)$ a martingale solution to SDE \eqref{SDE} with initial distribution $\mu_{0}$ if $\mathbb{P}_{\mu_0} \circ X_{0}^{-1}=\mu_{0}$, and for every $f \in C_{b}^{1, 2}\left([0, T]\times\mathbb{R}^{d}\right)$, $M_{t}^{f}$ is a $\mathcal{B}_{t}$-martingale under $\mathbb{P}_{\mu_0}$, where
\begin{equation*}
M_{t}^{f}:=f\left(t, X_{t}\right)-f\left(0, X_{0}\right)-\int_{0}^{t} (\partial_sf+\mathscr{L}f)\left(s, X_{s}\right) \mathrm{d}s.
\end{equation*}
\end{definition}

\begin{definition}
{(Weak solution)}
Let $\mu_{0} \in \mathcal{P}\left(\mathbb{R}^{d}\right)$. The SDE \eqref{SDE} is said to have a weak solution with initial law $\mu_{0}$ if there exists a filtered probability space $(\Omega, \mathscr{F}_{t}, (\mathscr{F}_{t})_{0 \leq t \leq T}, P)$, on which there are defined an $\mathscr{F}_{t}$-adapted continuous process $X:=(X_{t})_{0 \leq t \leq T}$ taking values in $\mathbb{R}^{d}$ and an $\mathscr{F}_{t}$-adapted Poisson martingale measure $\tilde{N}(\mathrm{d} t, \mathrm{d} z)$ with the given compensator $\mathrm{d} t \nu(\mathrm{d} z)$ such that, $X_{0}$ is distributed as $\mu_{0}$ and a.s.,
\begin{equation}\label{SDE1}
X_{t}=X_{0}+\int_{0}^{t} b_{s}\left(X_{s}\right) \mathrm{d} s+\int_{0}^{t}\int_{|z|<R}g_{s}\left(X_{s-}, z\right) \tilde{N}(\mathrm{d} s, \mathrm{d}z), \quad \text{ for all } t \in[0, T] .
\end{equation}
We denote this solution by $\left(\Omega, \mathscr{F}_{t}, (\mathscr{F}_{t})_{0 \leq t \leq T}, P ; X, \tilde{N}\right)$.
\end{definition}

The following result states the equivalence between weak solutions and martingale solutions, as in \cite[Theorem 5.6]{XZ13}.
\begin{proposition}\label{prop1}
The existence of martingale solutions implies the existence of weak solutions. In particular, the uniqueness of weak solutions implies the uniqueness of martingale solutions.
\end{proposition}

We also need the following extension of Figalli's results \cite{AF08}. For the following superposition principle, see \cite{FX19}.
\begin{proposition}\label{prop2}
Let $\mu_{t} \in \mathcal{P}\left(\mathbb{R}^{d}\right)$ be a measure-valued weak solution to equation \eqref{FPE} with initial value $\mu_{0} \in \mathcal{P}\left(\mathbb{R}^{d}\right)$, that is, it satisfies \eqref{def1.1} and \eqref{def1.2}. Then, there exists a martingale solution $P_{\mu_{0}}$ to SDE \eqref{SDE} with initial law $\mu_{0}$ such that for all $t \in[0, T]$, and for all $\varphi \in C_{c}^{\infty}\left(\mathbb{R}^{d}\right)$,
\begin{equation*}
\int_{\mathbb{R}^{d}} \varphi(x) \mathrm{d} \mu_{t}(x)=\int_{\mathbb{D}_T} \varphi\left(w_{t}\right) \mathrm{d} P_{\mu_{0}}(w).
\end{equation*}
\end{proposition}

This following result can be proved by following the ideas in the proof of  \cite[Proposition 3.6]{LL19}.
\begin{proposition}\label{prop3}
For each $i \in\{1,2\}$, let $\left(\Omega^{i}, \mathscr{F}^{i},\left(\mathscr{F}_{t}^{i}\right)_{0 \leq t \leq T}, P^{i} ; X^{i}, \tilde{N}\right)$ be a weak solution to SDE \eqref{SDE} with corresponding coefficients $b^{i}$ and $g^{i}$, having the initial law $\mu_{0}^{i} \in \mathcal{P}\left(\mathbb{R}^{d}\right) .$ Then, for any $\pi \in \mathcal{C}\left(\mu_{0}^{1}, \mu_{0}^{2}\right)$, there exist a filtered probability space $\left(\Omega, \mathscr{F},\left(\mathscr{F}_{t}\right)_{0 \leq t \leq T}, P\right)$ and two $\mathbb{R}^{d}$-valued $\mathscr{F}_{t}$-adapted continuous processes $Y^{1}$ and $Y^{2}$, such that $\left(Y_{0}^{1}, Y_{0}^{2}\right)$ is distributed as $\pi$, and for each $i \in\{1,2\}$,  \\
(1) $\left(X^{i}, \tilde{N}\right)$ and $\left(Y^{i}, \tilde{N}\right)$ have the same distribution;\\
(2) $\left(\Omega, \mathscr{F}, \left(\mathscr{F}_{t}\right)_{0 \leq t \leq T}, P ; Y^{i}, \tilde{N}\right)$ is a weak solution of SDE \eqref{SDE} with coefficients $b^{i}$ and $g^{i}$.
\end{proposition}

\begin{lemma}
 Let $\left(\mu_{t}\right)_{t \in[0, T]}$ be a solution to \eqref{FPE}.\\
(1) If $\int_{\mathbb{R}^{d}}|x| \mathrm{d} \mu_{0}(x)<+\infty$, then $\sup _{0 \leq t \leq T} \int_{\mathbb{R}^{d}}|x| \mathrm{d} \mu_{t}(x)<+\infty$;\\
(2) If $\int_{\mathbb{R}^{d}} \log \left(1+|x|^{2}\right) \mathrm{d} \mu_{0}(x)<+\infty$, then $\sup _{0 \leq t \leq T} \int_{\mathbb{R}^{d}} \log \left(1+|x|^{2}\right) \mathrm{d} \mu_{t}(x)<+\infty$.
\end{lemma}
\begin{proof}
According to Propositions \ref{prop1} and \ref{prop2}, if $\left(\mu_{t}\right)_{t \in[0, T]}$ is a weak solution to the non-local Fokker-Planck equation \eqref{FPE} with initial value $\mu_{0}$, then there exists a weak solution $\left(\Omega, \mathscr{F},\left(\mathscr{F}_{t}\right)_{0 \leq t \leq T}, P ; X, \tilde{N}\right)$ to the SDE \eqref{SDE} such that law$\left(X_{0}\right)=\mu_{0}$ and
\eqref{SDE1} holds a.s.\\
(i) Since the solution $\left(\mu_{t}\right)_{0 \leq t \leq T}$ of the non-local Fokker-Planck equation \eqref{FPE} satisfies \eqref{def1.1}, the stochastic integral in \eqref{SDE1} is a square integrable martingale. Moreover, $\mathbb{E}\left|\int_{0}^{t}\int_{|z|<R}g_{s}\left(X_{s-}, z\right) \tilde{N}(\mathrm{d}s, \mathrm{d}z)\right|^2=\mathbb{E} \int_{0}^{t}\int_{|z|<R}\left|g_{s}(X_{s-}, z)\right|^2 \nu(\mathrm{d}z) \mathrm{d}s$ (see \cite[Lemma 2.4]{K04}). By the Cauchy-Schwarz inequality and the isometry formula,
\begin{equation*}
\begin{aligned}
\mathbb{E}\left|X_{t}\right| & \leq \mathbb{E}\left|X_{0}\right|+\mathbb{E}\left|\int_{0}^{t} b_{s}\left(X_{s}\right) \mathrm{d} s\right|+\mathbb{E}\left|\int_{0}^{t}\int_{|z|<R}g_{s}\left(X_{s-}, z\right) \tilde{N}(\mathrm{d} s, \mathrm{d}z)\right|\\
& \leq \int_{\mathbb{R}^{d}}|x| \mathrm{d} \mu_{0}(x)+\int_{0}^{t} \mathbb{E}\left|b_{s}\left(X_{s}\right)\right| \mathrm{d}s+\left[\mathbb{E} \int_{0}^{t}\int_{|z|<R}\left|g_{s}(X_{s-}, z)\right|^2 \nu(\mathrm{d}z) \mathrm{d}s\right]^{1 / 2} \\
& \leq \int_{\mathbb{R}^{d}}|x| \mathrm{d} \mu_{0}(x)+\int_{0}^{T} \int_{\mathbb{R}^{d}}\left|b_{s}(x)\right| \mathrm{d} \mu_{s}(x) \mathrm{d}s+\left[\int_{0}^{T} \int_{\mathbb{R}^{d}}\int_{|z|<R}\left|g_{s}(x, z)\right|^2 \nu(\mathrm{d}z) \mathrm{d} \mu_{s}(x) \mathrm{d} s\right]^{1 / 2}.
\end{aligned}
\end{equation*}
which is finite by \eqref{def1.1}. Therefore,
\begin{equation*}
\sup _{0 \leq t \leq T} \int_{\mathbb{R}^{d}}|x| \mathrm{d} \mu_{t}(x)=\sup _{0 \leq t \leq T} \mathbb{E}\left|X_{t}\right|<\infty.
\end{equation*}
(ii) By the It\^o formula,
\begin{equation*}
\begin{aligned}
\log \left(\left|X_{t}\right|^{2}+1\right)=&\log \left(\left|X_{0}\right|^{2}+1\right)+2\int_{0}^{t}\frac{\left\langle X_{s}, b_{s}(X_{s})\right\rangle}{\left|X_{s}\right|^{2}+1} \mathrm{d}s\\
&+\int_{0}^{t} \int_{|z|<R} \log \frac{|X_{s-}+g_s(X_{s-}, z)|^{2}+1}{|X_{s-}|^{2}+1} \tilde{N}(\mathrm{d}z, \mathrm{d}s)\\
&+\int_{0}^{t} \int_{|z|<R}\left(\log \frac{|X_{s-}+g_s(X_{s-}, z)|^{2}+1}{|X_{s-}|^{2}+1}-\frac{2\langle X_{s-}, g_s(X_{s-}, z)\rangle}{|X_{s-}|^{2}+1}\right) \nu(\mathrm{d}z) \mathrm{d}s.
\end{aligned}
\end{equation*}
The quadratic variation of the martingale part can be estimated as
\begin{equation*}
\begin{aligned}
& \mathbb{E} \int_{0}^{T} \int_{|z|<R} \left(\log \frac{|X_{t-}+g_t(X_{t-}, z)|^{2}+1}{|X_{t-}|^{2}+1}\right)^2 \nu(\mathrm{d}z)\mathrm{d}t \\
\leq & C \mathbb{E} \int_{0}^{T}\int_{|z|<R} \left|g_t(X_{t-}, z)\right|^{2} \nu(\mathrm{d}z)\mathrm{d}t \\
\leq & C \int_{0}^{T} \int_{\mathbb{R}^{d}}\int_{|z|<R}\left|g_{t}(x, z)\right|^2 \nu(\mathrm{d}z) \mathrm{d} \mu_{t}(x) \mathrm{d}t<\infty.
\end{aligned}
\end{equation*}
Then, we have
\begin{equation*}
\begin{aligned}
\mathbb{E}\log \left(\left|X_{t}\right|^{2}+1\right) \leq & \mathbb{E}\log \left(\left|X_{0}\right|^{2}+1\right)+2\mathbb{E}\int_{0}^{t}\frac{|\left\langle X_{s}, b_{s}(X_{s})\right\rangle|}{\left|X_{s}\right|^{2}+1} \mathrm{d}s\\
&+\mathbb{E}\int_{0}^{t} \int_{|z|<R}\left|\log \frac{|X_{s-}+g_s(X_{s-}, z)|^{2}+1}{|X_{s-}|^{2}+1}-\frac{2\langle X_{s-}, g_s(X_{s-}, z)\rangle}{|X_{s-}|^{2}+1}\right| \nu(\mathrm{d}z) \mathrm{d}s.
\end{aligned}
\end{equation*}
Let $k(x)=\log \left(x^2+1\right)$. By elementary calculations and Taylor's expansion, one has
\begin{equation*}
\left|k(x+y)-k(x)-y^{i} \partial_{i} k(x)\right| \leq \left|y^{i} y^{j} \partial_{i} \partial_{j} k\left(x+\theta_{1} y\right)\right|/2 \leq C|y|^2,
\end{equation*}
where C is a positive constant and $\theta_1 \in (0, 1)$. Thus, we obtain
\begin{equation*}
\begin{aligned}
\mathbb{E}\log \left(\left|X_{t}\right|^{2}+1\right) \leq &\mathbb{E}\log \left(\left|X_{0}\right|^{2}+1\right)+\int_{0}^{t} \mathbb{E}\left|b_{s}\left(X_{s}\right)\right| \mathrm{d}s+C\mathbb{E} \int_{0}^{t}\int_{|z|<R}\left|g_{s}(X_{s-}, z)\right|^2 \nu(\mathrm{d}z) \mathrm{d}s \\
\leq & \int_{\mathbb{R}^{d}} \log \left(1+|x|^{2}\right) \mathrm{d} \mu_{0}(x)+\int_{0}^{T} \int_{\mathbb{R}^{d}}\left|b_{s}(x)\right| \mathrm{d} \mu_{s}(x) \mathrm{d}s\\
&+C\int_{0}^{T} \int_{\mathbb{R}^{d}}\int_{|z|<R}\left|g_{s}(x, z)\right|^2 \nu(\mathrm{d}z) \mathrm{d} \mu_{s}(x) \mathrm{d}s.
\end{aligned}
\end{equation*}
This immediately implies the desired result.
\end{proof}

\section{Proof of the  Kantorovich-Rubinstein distance estimate }\label{sec4}

In this section, we present the proof of Theorem \ref{theo}.
\begin{proof}
By Proposition \ref{prop2}, there exists a martingale solution $P_{\mu_{0}^{i}}^{i}$ to SDE \eqref{SDE} with coefficients $b^{i}$ and $g^{i}$, $i=1,2$, and the initial probability distribution $\mu_{0}^{i}$ such that, for all $\varphi \in C_{c}^{\infty}\left(\mathbb{R}^{d}\right)$,
\begin{equation*}
\int_{\mathbb{R}^{d}} \varphi(x) \rho_{t}^{i}(x) \mathrm{d} x=\int_{\mathbb{D}_{T}} \varphi\left(w_{t}\right) \mathrm{d} P_{\mu_{0}^{i}}^{i}(w).
\end{equation*}
Applying Proposition \ref{prop1}, we obtain a weak solution $\left(\Omega^{i}, \mathscr{F}^{i},\left(\mathscr{F}_{t}^{i}\right)_{0 \leq t \leq T}, P^{i} ; X^{i}, \tilde{N}\right)$ to SDE \eqref{SDE} with coefficients $b^{i}$ and $g^{i}$, satisfying $P^i\circ\left(X^{i}\right)^{-1}=P_{\mu_{0}^{i}}^{i}$. Next, we can find $\pi_{\delta} \in \mathcal{C}\left(\mu_{0}^{1}, \mu_{0}^{2}\right)$ such that
\begin{equation}\label{dist}
\tilde{\mathcal{D}}_{\delta}\left(\mu_{0}^{1}, \mu_{0}^{2}\right)=\int_{\mathbb{R}^{d} \times \mathbb{R}^{d}} \log \left(1+\frac{|x-y|^{2}}{\delta^{2}}\right) \mathrm{d} \pi_{\delta}(x, y).
\end{equation}
Thus, it follows from Proposition \ref{prop3} that there exists a common filtered probability space $\left(\Omega, \mathscr{F}, \left(\mathscr{F}_{t}\right)_{0 \leq t \leq T}, P\right)$, on which there are defined two continuous $\mathscr{F}_{t}$-adapted processes $Y^{1}$ and $Y^{2}$ such that law$\left(Y_{0}^{1}, Y_{0}^{2}\right)=\pi_{\delta}$, and for $i=1,2$, $Y^{i}$ is distributed as $P_{\mu_{0}^{i}}^{i}$ on $\mathbb{D}_{T}$; moreover, it holds a.s. that
\begin{equation*}
Y_{t}^{i}=Y_{0}^{i}+\int_{0}^{t} b_{s}^{i}\left(Y_{s}^{i}\right) \mathrm{d} s+\int_{0}^{t}\int_{|z|<R}g^i_{s}\left(Y^i_{s-}, z\right) \tilde{N}(\mathrm{d}s, \mathrm{d}z), \quad \text { for all } t \in[0, T].
\end{equation*}

Define $Z_t=Y^1_t -Y^2_t$, $G_t(z)=g^1_t(Y^1_{t-}, z) - g^2_t(Y^2_{t-}, z)$ and fix $\delta > 0$. Then, applying It\^o's formula, we obtain
\begin{equation}\label{Ito1}
\begin{aligned}
\log \left(\frac{\left|Z_{t}\right|^{2}}{\delta^{2}}+1\right)=&\log \left(\frac{\left|Z_{0}\right|^{2}}{\delta^{2}}+1\right)+2\int_{0}^{t}\frac{\left\langle Z_{s}, b_{s}^{1}\left(Y_{s}^{1}\right)-b_{s}^{2}\left(Y_{s}^{2}\right)\right\rangle}{\left|Z_{s}\right|^{2}+\delta^{2}} \mathrm{d}s\\
&+\int_{0}^{t} \int_{|z|<R} \log \frac{|Z_{s-}+G_s(z)|^{2}+\delta^{2}}{|Z_{s-}|^{2}+\delta^{2}} \tilde{N}(\mathrm{d}z, \mathrm{d}s)\\
&+\int_{0}^{t} \int_{|z|<R}\left(\log \frac{|Z_{s-}+G_s(z)|^{2}+\delta^{2}}{|Z_{s-}|^{2}+\delta^{2}}-\frac{2Z_{s-}^*G_s(z)}{|Z_{s-}|^{2}+\delta^{2}}\right) \nu(\mathrm{d}z) \mathrm{d}s.
\end{aligned}
\end{equation}
Let $h(x)=\log \left(\frac{\left|x\right|^{2}}{\delta^{2}}+1\right)$. By elementary calculations, one has
\begin{equation*}
\left|\partial_{i} h(x)\right| = \frac{2x_i}{\delta^{2}+|x|^2}\leqslant C,
\end{equation*}
where C is a positive constant.
On the other hand, by Taylor's expansion, we have
\begin{equation}\label{Tay}
|h(x+y)-h(x)| \leqslant\left|y^{i} \partial_{i} h\left(x+\theta_2 y\right)\right| \leqslant C|y|,
\end{equation}
where $\theta_2 \in (0, 1)$.
The quadratic variation of the martingale part on $[0, T]$ is finite, since, by \eqref{Tay}
\begin{equation*}
\begin{aligned}
& \mathbb{E} \int_{0}^{T} \int_{|z|<R} \left(\log \frac{|Z_{t-}+G_t(z)|^{2}+\delta^{2}}{|Z_{t-}|^{2}+\delta^{2}}\right)^2 \nu(\mathrm{d}z)\mathrm{d}t \\
& \leq C \mathbb{E} \int_{0}^{T}\int_{|z|<R} \left|g_{t}^{1}(Y_{t-}^{1}, z) - g_{t}^{2}(Y_{t-}^{2}, z)\right|^{2} \nu(\mathrm{d}z)\mathrm{d}t \\
& \leq C \int_{0}^{T}\int_{|z|<R} \left(\int_{\mathbb{R}^{d}}\left|g_{t}^{1}(x, z)\right|^2 \mathrm{d} \mu_{t}^{1}(x)+\int_{\mathbb{R}^{d}}\left|g_{t}^{2}(x, z)\right|^2 \mathrm{d} \mu_{t}^{2}(x)\right) \nu(\mathrm{d}z)\mathrm{d} t<\infty.
\end{aligned}
\end{equation*}
Hence, the second line of the right-hand side of \eqref{Ito1} is a continuous martingale. Taking expectation on both sides of \eqref{Ito1} with respect to $P$ yields
\begin{equation*}
\begin{aligned}
\mathbb{E} \log \left(\frac{\left|Z_{t}\right|^{2}}{\delta^{2}}+1\right) & \leq \mathbb{E} \log \left(\frac{\left|Z_{0}\right|^{2}}{\delta^{2}}+1\right)+2 \mathbb{E} \int_{0}^{t} \frac{\left\langle Z_{s}, b_{s}^{1}\left(Y_{s}^{1}\right)-b_{s}^{2}\left(Y_{s}^{2}\right)\right\rangle}{\left|Z_{s}\right|^{2}+\delta^{2}} \mathrm{d} s \\
&+\mathbb{E}\int_{0}^{t} \int_{|z|<R}\left(\log \frac{|Z_{s-}+G_s(z)|^{2}+\delta^{2}}{|Z_{s-}|^{2}+\delta^{2}}-\frac{2Z_{s-}^*G_s(z)}{|Z_{s-}|^{2}+\delta^{2}}\right) \nu(\mathrm{d}z) \mathrm{d}s.
\end{aligned}
\end{equation*}
We notice that the joint distribution of $\left(Y_{t}^{1}, Y_{t}^{2}\right)$ belongs to $\mathcal{C}(\mu_{t}^{1}, \mu_{t}^{2})$, and deduce from \eqref{dist} that
\begin{equation}\label{I}
\begin{aligned}
\tilde{\mathcal{D}}_{\delta}\left(\mu_{t}^{1}, \mu_{t}^{2}\right) & \leq \tilde{\mathcal{D}}_{\delta}\left(\mu_{0}^{1}, \mu_{0}^{2}\right) +2 \mathbb{E} \int_{0}^{t} \frac{\left\langle Z_{s}, b_{s}^{1}\left(Y_{s}^{1}\right)-b_{s}^{2}\left(Y_{s}^{2}\right)\right\rangle}{\left|Z_{s}\right|^{2}+\delta^{2}} \mathrm{d} s \\
&+\mathbb{E} \int_{0}^{t} \int_{|z|<R}\left(\log \frac{|Z_{s-}+G_s(z)|^{2}+\delta^{2}}{|Z_{s-}|^{2}+\delta^{2}}-\frac{2Z_{s-}^*G_s(z)}{|Z_{s-}|^{2}+\delta^{2}}\right) \nu(\mathrm{d}z) \mathrm{d}s \\
&=: \tilde{\mathcal{D}}_{\delta}\left(\mu_{0}^{1}, \mu_{0}^{2}\right)+I_{1}+I_{2} .
\end{aligned}
\end{equation}
Below we shall estimate the two terms $I_{1}$ and $I_{2}$ separately.

\begin{description}
\item[An estimate of $I_{1}$:]
\end{description}

By the triangle inequality, we obtain
\begin{equation}
\begin{aligned}
I_{1} & \leq 2 \mathbb{E} \int_{0}^{t} \frac{\left|b_{s}^{1}\left(Y_{s}^{1}\right)-b_{s}^{2}\left(Y_{s}^{2}\right)\right|}{\sqrt{\left|Z_{s}\right|^{2}+\delta^{2}}} \mathrm{~d} s \\
& \leq 2 \mathbb{E} \int_{0}^{t} \frac{\left|b_{s}^{1}\left(Y_{s}^{1}\right)-b_{s}^{1}\left(Y_{s}^{2}\right)\right|}{\sqrt{\left|Z_{s}\right|^{2}+\delta^{2}}} \mathrm{~d} s+2 \mathbb{E} \int_{0}^{t} \frac{\left|b_{s}^{1}\left(Y_{s}^{2}\right)-b_{s}^{2}\left(Y_{s}^{2}\right)\right|}{\sqrt{\left|Z_{s}\right|^{2}+\delta^{2}}} \mathrm{~d} s \\
&=: I_{1,1}+I_{1,2}.
\end{aligned}
\end{equation}
First, we estimate the term $I_{1,2}$. Noticing that $Y_{s}^{2}$ is distributed as $\rho_{s}^{2}(x) \mathrm{d} x .$ Thus, using the H\"older's inequality, we have
\begin{equation}\label{I12}
\begin{aligned}
I_{1,2} & \leq \frac{2}{\delta} \int_{0}^{t} \int_{\mathbb{R}^{d}}\left|b_{s}^{1}(x)-b_{s}^{2}(x)\right| \rho_{s}^{2}(x) \mathrm{d} x \mathrm{d} s \\
& \leq \frac{2}{\delta} \int_{0}^{t}\left\|b_{s}^{1}-b_{s}^{2}\right\|_{L^{p}}\left\|\rho_{s}^{2}\right\|_{L^{q}} \mathrm{d} s \\
& \leq \frac{2}{\delta}\left\|\rho^{2}\right\|_{L^{\infty}\left(L^{q}\right)}\left\|b^{1}-b^{2}\right\|_{L^{1}\left(L^{p}\right)}.
\end{aligned}
\end{equation}
Next, we estimate the term $I_{1,1}$. Let $\chi \in C_{c}^{\infty}\left(\mathbb{R}^{d}, \mathbb{R}_{+}\right)$ such that $\operatorname{supp}(\chi) \subset B(1)$ and $\int_{\mathbb{R}^{d}} \chi(x) \mathrm{d} x=1$. Let $\chi_{\varepsilon}(x):=\varepsilon^{-d} \chi(x / \varepsilon)$, $\varepsilon \in(0,1)$. Define $b_{s}^{1, \varepsilon}:=b_{s}^{1} * \chi_{\varepsilon}.$ Then for a.e. $s \in[0, T]$, $b_{s}^{1, \varepsilon} \in C^{\infty}\left(\mathbb{R}^{d}\right)$
for every $\varepsilon \in(0,1).$ Then, by the triangle inequality,
\begin{equation}
\begin{aligned}
I_{1,1} & \leq 2 \mathbb{E} \int_{0}^{t} \frac{\left|b_{s}^{1, \varepsilon}\left(Y_{s}^{1}\right)-b_{s}^{1, \varepsilon}\left(Y_{s}^{2}\right)\right|}{\sqrt{\left|Z_{s}\right|^{2}+\delta^{2}}} \mathrm{d} s \\
&+2 \mathbb{E} \int_{0}^{t} \frac{\left|b_{s}^{1, \varepsilon}\left(Y_{s}^{1}\right)-b_{s}^{1}\left(Y_{s}^{1}\right)\right|+\left|b_{s}^{1, \varepsilon}\left(Y_{s}^{2}\right)-b_{s}^{1}\left(Y_{s}^{2}\right)\right|}{\sqrt{\left|Z_{s}\right|^{2}+\delta^{2}}} \mathrm{d} s \\
&=: I_{1,1,1}+I_{1,1,2} .
\end{aligned}
\end{equation}
As an application of (1) of Lemma \ref{lem}, we have
\begin{equation*}
\begin{aligned}
I_{1,1,1} & \leq 2 C_{d} \mathbb{E} \int_{0}^{t}\left(M\left|\nabla b_{s}^{1, \varepsilon}\right|\left(Y_{s}^{1}\right)+M\left|\nabla b_{s}^{1, \varepsilon}\right|\left(Y_{s}^{2}\right)\right) \mathrm{d} s \\
&=2 C_{d} \int_{0}^{t} \int_{\mathbb{R}^{d}} M\left|\nabla b_{s}^{1, \varepsilon}\right|(x)\left(\rho_{s}^{1}(x)+\rho_{s}^{2}(x)\right) \mathrm{d}x \mathrm{d} s.
\end{aligned}
\end{equation*}
Then, using the H\"older's inequality and (2) of Lemma \ref{lem}, we can estimate
\begin{equation}
\begin{aligned}
I_{1,1,1} & \leq 2 C_{d} \int_{0}^{t}\left\|M\left|\nabla b_{s}^{1, \varepsilon}\right|\right\|_{L^{p}}\left(\left\|\rho_{s}^{1}\right\|_{L^{q}}+\left\|\rho_{s}^{2}\right\|_{L^{q}}\right) \mathrm{d} s \\
& \leq 2 C_{d}\left(\sum_{i=1}^{2}\left\|\rho^{i}\right\|_{L^{\infty}\left(L^{q}\right)}\right) \int_{0}^{t} C_{d, p}\left\|\nabla b_{s}^{1, \varepsilon}\right\|_{L^{p}} \mathrm{d} s. \\
& \leq C_{d, p}\left(\sum_{i=1}^{2}\left\|\rho^{i}\right\|_{L^{\infty}\left(L^{q}\right)}\right)\left\|\nabla b^{1, \varepsilon}\right\|_{L^{1}\left(L^{p}\right)}.
\end{aligned}
\end{equation}
The quantity $I_{1,1,2}$ can be estimated as follows:
\begin{equation*}
\begin{aligned}
I_{1,1,2} & \leq \frac{2}{\delta} \mathbb{E} \int_{0}^{t}\left(\left|b_{s}^{1, \varepsilon}\left(Y_{s}^{1}\right)-b_{s}^{1}\left(Y_{s}^{1}\right)\right|+\left|b_{s}^{1, \varepsilon}\left(Y_{s}^{2}\right)-b_{s}^{1}\left(Y_{s}^{2}\right)\right|\right) \mathrm{d} s \\
&=\frac{2}{\delta} \int_{0}^{t} \int_{\mathbb{R}^{d}}\left|b_{s}^{1, \varepsilon}(x)-b_{s}^{1}(x)\right|\left(\rho_{s}^{1}(x)+\rho_{s}^{2}(x)\right) \mathrm{d} x \mathrm{d} s.
\end{aligned}
\end{equation*}
Again, we use H\"older's inequality:
\begin{equation}\label{I112}
\begin{aligned}
I_{1,1,2} & \leq \frac{2}{\delta} \int_{0}^{t}\left\|b_{s}^{1, \varepsilon}-b_{s}^{1}\right\|_{L^{p}}\left(\left\|\rho_{s}^{1}\right\|_{L^{q}}+\left\|\rho_{s}^{2}\right\|_{L^{q}}\right) \mathrm{d} s \\
& \leq \frac{2}{\delta}\left(\sum_{i=1}^{2}\left\|\rho^{i}\right\|_{L^{\infty}\left(L^{q}\right)}\right) \int_{0}^{t}\left\|b_{s}^{1, \varepsilon}-b_{s}^{1}\right\|_{L^{p}} \mathrm{d} s .
\end{aligned}
\end{equation}
We notice that the right hand side of \eqref{I112} vanishes as $\varepsilon \rightarrow 0$, since $b^{1} \in L^{1}\left(0, T; L^{p}\left(\mathbb{R}^{d}, \mathbb{R}^{d}\right)\right)$. Collecting all these estimates, and letting $\varepsilon \rightarrow 0$, we arrive at
\begin{equation}\label{I11}
I_{1,1} \leq C_{d, p}\left(\sum_{i=1}^{2}\left\|\rho^{i}\right\|_{L^{\infty}\left(L^{q}\right)}\right)\left\|\nabla b^{1}\right\|_{L^{1}\left(L^{p}\right)}.
\end{equation}
We deduce from \eqref{I12} and \eqref{I11} that
\begin{equation}\label{I1}
I_{1} \leq \frac{2}{\delta}\left\|\rho^{2}\right\|_{L^{\infty}\left(L^{q}\right)}\left\|b^{1}-b^{2}\right\|_{L^{1}\left(L^{p}\right)}+C_{d, p}\left(\sum_{i=1}^{2}\left\|\rho^{i}\right\|_{L^{\infty}\left(L^{q}\right)}\right)\left\|\nabla b^{1}\right\|_{L^{1}\left(L^{p}\right)}.
\end{equation}

\begin{description}
\item[An estimate of $I_{2}$:]
\end{description}

We estimate the term $I_{2}$ in the same way as for $I_{1}$. First, by the elementary inequality
\begin{equation}\nonumber
|\log(1+r)-r|\leq C|r|^2, \quad r \geqslant -\frac{1}{2},
\end{equation}
we have
\begin{equation}\nonumber
\begin{aligned}
I_{2} & \leq \mathbb{E} \int_{0}^{t} \int_{|z|<R}\left|\log \left(1+\frac{|Z_{s-}+G_s(z)|^{2}-|Z_{s-}|^2}{|Z_{s-}|^{2}+\delta^{2}}\right)-\frac{|Z_{s-}+G_s(z)|^{2}-|Z_{s-}|^2}{|Z_{s-}|^{2}+\delta^{2}}\right| \nu(\mathrm{d}z) \mathrm{d}s  \\
       &+  \mathbb{E} \int_{0}^{t} \int_{|z|<R}\frac{|G_s(z)|^{2}}{|Z_{s-}|^{2}+\delta^{2}}\nu(\mathrm{d}z) \mathrm{d}s  \\
       & \leq C \mathbb{E} \int_{0}^{t} \int_{|z|<R} \left(\frac{2\langle Z_{s-}, G_s(z)\rangle+|G_s(s)|^2}{|Z_{s-}|^{2}+\delta^{2}}\right)^2 \nu(\mathrm{d}z) \mathrm{d}s+ \mathbb{E} \int_{0}^{t} \int_{|z|<R}\frac{|G_s(z)|^{2}}{|Z_{s-}|^{2}+\delta^{2}}\nu(\mathrm{d}z) \mathrm{d}s\\
       & \leq C \mathbb{E} \int_{0}^{t} \int_{|z|<R} \frac{8\langle Z_{s-}, G_s(z)\rangle^2+2|G_s(s)|^4}{(|Z_{s-}|^{2}+\delta^{2})^2} \nu(\mathrm{d}z) \mathrm{d}s+ \mathbb{E} \int_{0}^{t} \int_{|z|<R}\frac{|G_s(z)|^{2}}{|Z_{s-}|^{2}+\delta^{2}}\nu(\mathrm{d}z) \mathrm{d}s\\
       & \leq C \mathbb{E} \int_{0}^{t} \int_{|z|<R} \left(\frac{8|G_s(z)|^2}{|Z_{s-}|^{2}+\delta^{2}}+\frac{2|G_s(s)|^4}{(|Z_{s-}|^{2}+\delta^{2})^2}\right) \nu(\mathrm{d}z) \mathrm{d}s+ \mathbb{E} \int_{0}^{t} \int_{|z|<R}\frac{|G_s(z)|^{2}}{|Z_{s-}|^{2}+\delta^{2}}\nu(\mathrm{d}z) \mathrm{d}s\\
       & \leq C \mathbb{E} \int_{0}^{t} \int_{|z|<R}\frac{|G_s(z)|^{2}}{|Z_{s-}|^{2}+\delta^{2}}\nu(\mathrm{d}z) \mathrm{d}s + C \mathbb{E} \int_{0}^{t} \int_{|z|<R}\frac{|G_s(z)|^{4}}{(|Z_{s-}|^{2}+\delta^{2})^2}\nu(\mathrm{d}z) \mathrm{d}s \\
&=: I_{2,1}+I_{2,2}.
\end{aligned}
\end{equation}
We first estimate the term $I_{2, 1}$. By the triangle inequality and inequality $(a+b)^2\leqslant 2a^2+2b^2$, we get
\begin{equation}\nonumber
\begin{aligned}
I_{2,1} &\leq C\mathbb{E} \int_{0}^{t} \int_{|z|<R} \frac{\left|g_{s}^{1}(Y_{s-}^{1}, z)-g_{s}^{1}(Y_{s-}^{2}, z)\right|^2}{|Z_{s-}|^{2}+\delta^{2}} \nu(\mathrm{d}z)\mathrm{d}s+C\mathbb{E} \int_{0}^{t} \int_{|z|<R} \frac{\left|g_{s}^{1}(Y_{s-}^{2}, z)-g_{s}^{2}(Y_{s-}^{2}, z)\right|^2}{|Z_{s-}|^{2}+\delta^{2}} \nu(\mathrm{d}z)\mathrm{d}s \\
&=: I_{2,1,1}+I_{2,1,2}.
\end{aligned}
\end{equation}
In the similar way as for $I_{1, 2}$, we obtain
\begin{equation}\label{I212}
I_{2,1,2} \leq \frac{C}{\delta^2}\|\rho^{2}\|_{L^{\infty}(L^{q})}\int_{|z|<R}\left\|g^{1}-g^{2}\right\|_{L^{2}\left(L^{2p}\right)}^{2}\nu(\mathrm{d}z).
\end{equation}
Next, defining $g_{s}^{1, \varepsilon}(x):=g_{s}^{1} (x)* \chi_{\varepsilon}(x)$, we have
\begin{equation*}
\begin{aligned}
I_{2,1,1} & \leq  C \mathbb{E} \int_{0}^{t} \int_{|z|<R} \frac{\left|g_{s}^{1, \varepsilon}(Y_{s}^1, z)-g_{s}^{1, \varepsilon}(Y_{s}^{2}, z)\right|^{2}}{|Z_{s-}|^{2}+\delta^{2}} \nu(\mathrm{d}z)\mathrm{d}s \\
&+C \mathbb{E} \int_{0}^{t} \int_{|z|<R} \frac{\left|g_{s}^{1, \varepsilon}(Y_{s}^{1}, z)-g_{s}^{1}(Y_{s}^{1}, z)\right|^{2}+\left|g_{s}^{1, \varepsilon}(Y_{s}^{2},z)-g_{s}^{1}(Y_{s}^{2}, z)\right|^{2}}{|Z_{s-}|^{2}+\delta^{2}} \nu(\mathrm{d}z)\mathrm{d}s \\
&=: I_{2,1,1,1}+I_{2,1,1,2} .
\end{aligned}
\end{equation*}
Following the arguments for $I_{1,1,1}$ and $I_{1,1,2}$, respectively, we can show that
\begin{equation}\nonumber
I_{2,1,1,1} \leq C \left(\sum_{i=1}^{2}\|\rho^{i}\|_{L^{\infty}(L^{q})}\right) \int_{|z|<R} \|\nabla g^{1, \varepsilon}\|_{L^{2}(L^{2 p})}^{2}\nu(\mathrm{d}z),
\end{equation}
and
\begin{equation*}
I_{2,1,1,2} \leq \frac{C}{\delta^{2}}\left(\sum_{i=1}^{2}\|\rho^{i}\|_{L^{\infty}(L^{q})}\right) \int_{0}^{t}\int_{|z|<R} \|g_{s}^{1, \varepsilon}-g_{s}^{1}\|_{L^{2 p}}^{2} \nu(\mathrm{d}z)\mathrm{d}s.
\end{equation*}
Combining the above two inequalities and letting $\varepsilon \rightarrow 0$, we arrive at
\begin{equation}\label{I211}
I_{2,1,1} \leq C \left(\sum_{i=1}^{2}\|\rho^{i}\|_{L^{\infty}(L^{q})}\right) \int_{|z|<R} \|\nabla g^{1}\|_{L^{2}(L^{2 p})}^{2}\nu(\mathrm{d}z).
\end{equation}
Combining \eqref{I212} and \eqref{I211}, we conclude that
\begin{equation}\label{I21}
\begin{aligned}
I_{2,1} & \leq  \frac{C}{\delta^{2}}\|\rho^{2}\|_{L^{\infty}(L^{q})} \int_{|z|<R} \|g^{1}-g^{2}\|_{L^{2}(L^{2 p})}^{2} \nu(\mathrm{d}z) \\
& + C^{\prime}\left(\sum_{i=1}^{2}\|\rho^{i}\|_{L^{\infty}(L^{q})}\right) \int_{|z|<R} \|\nabla g^{1}\|_{L^{2}(L^{2 p})}^{2}\nu(\mathrm{d}z).
\end{aligned}
\end{equation}
Treating the term $I_{2,2}$ in a similar way, we obtain
\begin{equation}\label{I22}
\begin{aligned}
I_{2,2} & \leq \frac{C}{\delta^{4}}\|\rho^{2}\|_{L^{\infty}(L^{q})} \int_{|z|<R} \|g^{1}-g^{2}\|_{L^{4}(L^{4 p})}^{4} \nu(\mathrm{d}z) \\
&+ C^{\prime}\left(\sum_{i=1}^{2}\|\rho^{i}\|_{L^{\infty}(L^{q})}\right) \int_{|z|<R} \|\nabla g^{1}\|_{L^{4}(L^{4 p})}^{4}\nu(\mathrm{d}z).
\end{aligned}
\end{equation}

Finally, substituting the estimates \eqref{I1}, \eqref{I21} and \eqref{I22} into \eqref{I}, we have proved that
\begin{equation}\nonumber
\begin{aligned}
\tilde{\mathcal{D}}_{\delta}\left(\mu_{t}^{1}, \mu_{t}^{2}\right) & \leq \tilde{\mathcal{D}}_{\delta}\left(\mu_{0}^{1}, \mu_{0}^{2}\right)+C\|\rho^{2}\|_{L^{\infty}(L^{q})}\left[ \frac{1}{\delta}\|b^{1}-b^{2}\|_{L^{1}(L^{p})}+\frac{1}{\delta^{2}} \int_{|z|<R} \|g^{1}-g^{2}\|_{L^{2}(L^{2 p})}^{2}\nu(\mathrm{d}z) \right.\\
&\left. +\frac{1}{\delta^{4}} \int_{|z|<R} \|g^{1}-g^{2}\|_{L^{4}(L^{4 p})}^{4}\nu(\mathrm{d}z) \right] +C^{\prime}\left(\sum_{i=1}^{2}\|\rho^{i}\|_{L^{\infty}(L^{q})}\right)\times \\
& \left[\|\nabla b^{1}\|_{L^{1}(L^{p})} + \int_{|z|<R} \|\nabla g^{1}\|_{L^{2}(L^{2 p})}^{2}\nu(\mathrm{d}z)+ \int_{|z|<R} \|\nabla g^{1}\|_{L^{4}(L^{4 p})}^{4}\nu(\mathrm{d}z) \right],
\end{aligned}
\end{equation}
where the constant $C, C^{\prime}>0$ depends only on $d$ and $p$. The proof is completed.
\end{proof}

\section{Examples}\label{sec5}
Let us consider the following example.

\begin{example}
Let $\{b^n\}_{n \geq 1}$ and $\{g^{n}\}_{n \geq 1}$ satisfy the assumptions of Theorem \ref{theo}. Suppose that a sequence of drifts $\{b^n\}_{n \geq 2}$ converges in $L^{1}\left(0, T ; L^{p}\left(\mathbb{R}^{d}\right)\right)$ to the limit drift $b^{1}$ and a sequence of jump diffusions $\left\{g^{n}(z)\right\}_{n \geq 2}$ converges in $L^{2}\left(0, T ; L^{2 p}(\mathbb{R}^{d})\right)\bigcap L^{4}\left(0, T ; L^{4 p}(\mathbb{R}^{d})\right)$
to the limit jump diffusion $g^{1}(z)$, for every z. Let $\mu_{t}^{n}=\rho_{t}^{n}(x) \mathrm{d} x$, with density $\rho_{t}^{n}$,  be the solution of  the nonlocal Fokker-Planck equation \eqref{FPE} corresponding to the coefficients $b^{n}$ and $g^{n}$, with the identical initial condition $\mu_{0}^{n}=\mu_{0}^{1}$. Furthermore, assume that the sequence of densities $\left\{\rho^{n}\right\}_{n \geq 1}$ is bounded in $L^{\infty}\left(0, T ; L^{1} \cap L^{q}\left(\mathbb{R}^{d}\right)\right)$. If we take $\delta$ in Theorem \ref{theo} to be the following value
\begin{equation}\nonumber
\begin{aligned}
\delta=\delta_{n}=\|b^{1}-b^{n}\|_{L^{1}(L^{p})}+\left(\int_{|z|<R} \|g^{1}-g^{n}\|_{L^{2}(L^{2 p})}^2\nu(\mathrm{d}z)\right)^{1/2}
+\left(\int_{|z|<R} \|g^{1}-g^{n}\|_{L^{4}(L^{4 p})}^4\nu(\mathrm{d}z)\right)^{1/4},
\end{aligned}
\end{equation}
then \eqref{theo1} of Theorem \ref{theo} implies that $\tilde{\mathcal{D}}_{\delta_{n}}\left(\mu_{t}^{1}, \mu_{t}^{n}\right) \leq \tilde{C}<\infty$. By Remark \ref{rem1}, we conclude that, as $n \rightarrow \infty, \mu_{t}^{n}$ converges weakly to $\mu_{t}^{1}$, with the speed of $\delta_{n}$.

This leads to a smooth approximation for non-local Fokker-Planck equations: Define $b^n:=b^{1} * \chi_{\varepsilon}$ and $g^n(z):=g^{1}(z) * \chi_{\varepsilon}$, where $\chi_{\varepsilon}(x):=\varepsilon^{-d} \chi(x / \varepsilon)$, $\varepsilon \in(0,1)$ and $\chi \in C_{c}^{\infty}\left(\mathbb{R}^{d}, \mathbb{R}_{+}\right)$ such that $\operatorname{supp}(\chi) \subset B(1)$ and $\int_{\mathbb{R}^{d}} \chi(x) \mathrm{d} x=1$. Then, we conclude that  a sequence of non-local Fokker-Planck equations with smooth coefficients $b^n$ and $g^n$ converges weakly to the non-local Fokker-Planck equation with weakly differentiable coefficients $b^1$ and $g^1$.

\end{example}

In this paper, we have examined the behavior of probability distributions for complex dynamical systems under non-Gaussian fluctuations. We have first estimated the Kantorovich-Rubinstein distance for solutions of non-local Fokker-Planck equations  associated with stochastic differential equations with non-Gaussian L\'evy noise. This is then  applied to establish  weak convergence of the corresponding probability distributions. Furthermore, this leads to a smooth approximation for  non-local Fokker-Planck equations, as illustrated in an example. For stochastic differential equation \eqref{SDE} with big jump, we can use the same method to estimate the Kantorovich-Rubinstein distance for solutions of corresponding non-local Fokker-Planck equations, but we need more assumptions about the jump coefficient.

\section*{Acknowledgements}
The authors are grateful to Li Lin and Li Lv for helpful discussions. This work was partly
supported by the NSFC grants 11771449.

\section*{Data Availability}
The data that support the findings of this study are available within the article.

\bibliographystyle{elsarticle-num}
\bibliography{refe}

\end{document}